\documentclass[reqno,10pt]{article}
\usepackage{amsmath}%
\usepackage{amsthm}
\usepackage{amsfonts}%
\usepackage{amssymb}%
\usepackage{graphicx}

\usepackage[comma,authoryear]{natbib}

\usepackage{tikz,ifthen}
\usetikzlibrary{positioning}
\usetikzlibrary{calc,through,intersections,shapes.geometric}
\usepackage{ifthen}
\usepackage{geometry}
\usepackage{pgfplots} 
\newtheorem{theorem}{Theorem}[]

\newtheorem{corollary}[theorem]{Corollary}

\newtheorem{lemma}[theorem]{Lemma}

\newtheorem{proposition}[theorem]{Proposition}

\theoremstyle{definition}

\newtheorem*{example}{Example}

\newcommand{\comments}[1]{}
\newcommand{\G}{{\mathcal G}}

\newcommand{\diag}{\mathop{\mathrm{diag}}\nolimits}

\begin{document}
\title{An upper bound on  the algebraic connectivity of regular graphs}

\author{Sera Aylin Cakiroglu\footnote{Current address: CRUK London Research Institute, 44 Lincoln's Inn Fields, WC2A 3LY London, UK}\\ School of Mathematical Sciences, Queen Mary University of London\\ Mile End Road London E1 4NS, UK\\ Email: s.cakiroglu@qmul.ac.uk\\ }

\maketitle

\section{Introduction}
The Laplacian matrix of a graph can be used to express different connectivity  measures of a graph. In the simplest way, the multiplicity of the   eigenvalue  to the all-$1$ eigenvector equals the number of connected components of the graph. Or,   most famously, it can be used to compute the number of spanning trees of a graph (or its connected components). The  eigenvalues of the Laplacian matrix play an important role in different areas of mathematics and, of course, also in graph theory. 
 The second smallest  eigenvalue is also known as the algebraic connectivity of the graph. This definition is due to \cite{Fiedler} who proved as one of the first results  that the vertex connectivity is an upper bound for the eigenvalue. The name  has its origin in the correspondence between the eigenvalue and the connectivity of the graph, that is it is equal to zero if and only if the graph is not connected.
There has been wide interest on the bounds on this eigenvalue (see for example \cite{Kirkland1,Kirkland2,Lu,Ghosh}),  whose importance is  `$[\dots]$ difficult to overemphasize', since the larger the algebraic connectivity of a graph $\G$, `$[\dots]$ the more difficult it is to cut $\G$ into pieces, and the more $\G$ expands' as \cite{Bollobas} writes.  

This gives the motivation for this paper where we derive a new upper bound for the algebraic connectivity of a regular graph  using the Higman-Sims technique introduced by \cite{Haemers}. We show that this bound is tight, by giving an example of a regular graph meeting this bound.
Together with a result on the connectivity of the neighbourhood graph of strongly regular graphs our result  gives  a characterization of a class of strongly regular graphs that maximize the algebraic connectivity among regular graphs which includes the complete multipartite graph.

\section{Preliminaries}

Let $\G$ be a graph given by a set of vertices $V(\G)$ and a set  $E(\G)$ of edges, i.e. two-element subsets of $V(\G)$.  We do not allow any loops or  multiple edges and all graphs we consider are assumed to be connected. 
The number of edges incident with a vertex $u\in V(\G)$ is called the \emph{degree} of $u$ and denoted by $\delta_u$.  If  the degrees of all vertices of $\G$ are all equal to a constant $\delta$, then  $\G$ is called $\delta$-regular. 
The adjacency matrix of a graph $\G$, that is the symmetric $|V(\G)|\times|V(\G)|$ matrix whose $ij$-entry is $1$ if  vertices $i$ and $j$ are joined by and edge and $0$ else, is denoted by $A(\G)$. We will denote the eigenvalues of the adjacency matrix by $\nu_1(\G)\geq\nu_2(\G)\ldots\geq\nu_v(\G)$.
Note that the row and column of $A(\G)$ corresponding to a vertex $u$ sum both to $\delta_u$ and if the graph is $\delta$-regular, then $\nu_1(\G)=\delta$.
The \emph{Laplacian matrix} of a graph with $v$ vertices is defined as
\[\Lambda(\G)=\diag(\delta_1,\ldots,\delta_v) {\mathbb I}_v-A(\G),\]
where ${\mathbb I}_v$ denotes the $v \times v$ identity matrix.

The all-one vector is always an eigenvector of $\Lambda(\G)$ with eigenvalue $0$ and its multiplicity is the number of connected components of $\G$. Due to this we call all other Laplacian eigenvalues \emph{non-trivial}. All non-trivial Laplacian eigenvalues are positive and strictly positive if $\G$ is connected.
The smallest non-trivial eigenvalue is also called the \emph{algebraic connectivity} of the graph.   Note that if $\G$ connected, the multiplicity of the eigenvalue $0$ is $1$ and therefore the algebraic connectivity is strictly positive.

For a vertex $u$ of $\G$, the \emph{neighbourhood graph} of $u$ is the graph  ${\mathcal G}_u$  with  vertex set 
\[ V(\G_u)=\{w\in V(\G)\setminus\{u\}|\{u,w\}\in E(\G)\}\] and edge set
\[E(\G_u)=\{f\in E(\G)|f=\{w_1,w_2\}, w_1,w_2\in V(\G_u)\}.\]

A \emph{strongly regular graph} with parameters $(\delta,\lambda,\mu)$ is a connected  $\delta$-regular graph such that each pair of adjacent vertices has exactly $\lambda$ common neighbours and  each pair of non-adjacent vertices has exactly $\mu\geq 1$ common neighbours.
The parameters $\lambda$ and $\mu$ of a strongly regular graph $\G$ on $v$ vertices and degree $\delta$ satisfy 
\begin{equation}\label{lemma parameters of SRG satisfy}
\delta(\delta-\lambda-1)=(v-\delta-1)\mu.
\end{equation}
The  eigenvalues of the adjacency matrix are $\delta$ with  multiplicity $1$ and 
\[
\nu_{1,2}(\G)=\frac{1}{2}\left[(\lambda-\mu)\pm\sqrt{(\lambda-\mu)^2+4(\delta-\mu)}\right]
\] 
with  multiplicities
\[
m_{1,2}=\frac{1}{2}\left[(v-1)\mp\frac{2\delta+(v-1)(\lambda-\mu)}{\sqrt{(\lambda-\mu)^2+4(\delta-\mu)}}\right].
\]

We will use the technique of using partitions of the graph's adjacency matrix to study subgraphs as introduced in \cite{Haemers}. At the heart of this lies the notion of interlacing eigenvalues.
Suppose $L$ is an  $m\times m$ matrix with $m\geq n$, having  eigenvalues $\nu_1(L)\geq \cdots\geq\nu_m(L)$. If
\[
\nu_i(L)\geq \nu_i(M)\geq \nu_{m-n+i}(L)
\]for all $i\in\{1,\ldots,n\}$, then we say that the eigenvalues of $M$ \emph{interlace} the eigenvalues of $L$. If there exists an integer $j\in\{0,\ldots, n\}$ such that
\[
\nu_i(L)=\nu_i(M)\text{ for }i=1,\ldots,j
\]and
\[
\nu_{m-n+i}(L)=\nu_i(M)\text{ for }i=j+1,\ldots,n,
\]then the interlacing is called \emph{tight}.

\begin{theorem}[\cite{Haemers}]\label{Corollary Haemers}
 Let $L$ be a symmetric real square matrix partitioned as follows
\[
L=\left(
\begin{array}[c]{ccc}
	L_{11}&\cdots& L_{1m}\\
	\vdots&      &\vdots\\
        L_{m1}&\cdots&L_{mm}
\end{array}
\right)
\]such that $L_{ii}$ is square for $i=1,\ldots,m$. Let $M$ be the $m\times m$ matrix whose $ij$-entry is the average row sum of $L_{ij}$ for $i,j=1,\ldots,m$.
\begin{enumerate}
 \item The eigenvalues of $M$ interlace the eigenvalues of $L$.
\item If the interlacing is tight, then $L_{ij}$ has constant row and column sums for $i,j=1,\ldots,m$.
\item If for all  $i,j\in\{1,\ldots,m\}$ the matrix $L_{ij}$ has constant row and column sums, then any eigenvalue of $M$ is also an eigenvalue of $L$ with not smaller a multiplicity.
\end{enumerate}
\end{theorem}

\section{A new bound on the algebraic connectivity of  regular graphs}

We will need the following lemma.

\begin{lemma}\label{Lemma Root}
For $x\in{\mathbb R}_{\geq 0}$, $1<\delta<v-1$ and $v\geq 3$
\[
(x(v-1)-\delta(\delta-1))^2+4(v-\delta-1)\delta(v-2\delta+x)> 0.
\]
\end{lemma}
\begin{proof}
For $x\geq 0$ and $2\delta\leq v$ the statement is of course true.

Now, let $v<2\delta$ and
\[
f(x)=(x(v-1)-\delta(\delta-1))^2+4(v-\delta-1)\delta(v-2\delta+x).
\]Then the derivative $f'(x)$
has   root
\[
x_0=\delta\left(\frac{\delta(v+1)-3(v-1)}{(v-1)^2}\right).
\]For $v<2\delta$  and $v\geq 3$ we have
\begin{eqnarray*}
\delta(v+1)-3(v-1)&>&\tfrac{v}{2}(v+1)-3(v-1)=\frac{(v-3)^2+v-3}{2}\geq 0.\\
\end{eqnarray*}
Therefore,  $x_0>0$ and since $f(x)$ is a quadratic polynomial with positive leading coefficient ,
$f(x)$ attains its minimum at $x_0$. Since $v>\delta-1$,
\[f(x_0)=\frac{4v \delta (v-\delta-1)^3 }{(v-1)^2}>0 \]
and the lemma follows.
\end{proof}

\begin{proposition}\label{Proposition lower bound eigenvalue}
Let $v\geq 3$ and let $F:{\mathbb R}_{\geq0}\rightarrow {\mathbb R}$,
\[
F(x)=\tfrac{x(v-1)-\delta(\delta-1)+\sqrt{(x(v-1)-\delta(\delta-1))^2+4(v-\delta-1)\delta(v-2\delta+x)}}{2(v-\delta-1)}.\]
Further, let ${\mathcal G}$ be a regular graph with degree $\delta$.  For a vertex $u\in V(\G)$ let ${\mathcal G}_u$ denote the neighbourhood graph. 
If ${\mathcal G}_u$ is not connected, let $\bar{\delta}_{{\mathcal C}(\G_u)}$ denote the  average degree of the connected component ${\mathcal C}(\G_u)$ and
\[
\eta_u=\max_{{\mathcal C}(\G_u)}\left\{\max\left\{\bar{\delta}_{{\mathcal C}(\G_u)},F(\bar{\delta}_{{\mathcal C}(\G_u)})\right\}\right\}.
\] 
If $\G_u$ is connected, let $\bar{\delta}_u$ denote the  average degree of $\G_u$ and
\[
\xi_u=F(\bar{\delta}_u).
\] 
Then 
\[
\varrho(\G)=\max_{u\in V(\G)}\left\{\{\eta_u|\G_u \text{ not connected }\}\cup \{\xi_u|\G_u\text{ connected }\}\right\}
\]
is a lower bound for $\nu_2(\G)$ and therefore $\delta-\varrho(\G)$ is an upper bound for the algebraic connectivity of $\G$.
\end{proposition}
\begin{proof}
\underline{Case 1: $\G_u$ is connected}

We divide the adjacency matrix of ${\mathcal G}$ into $9$ block matrices $M_{ij}$ according to $u$ and the vertices of $\G_u$ and  ${\mathcal G}\setminus \{u,\G_u\}$. 
With $\bar{\delta}_u$ denoting the average degree of  $\G_u$, the matrix of average row sums is
\[
M=\left(
\begin{array}[c]{ccc}
	0&  \delta& 0\\
	1&\bar{\delta}_{\G_u}&\delta-\bar{\delta}_{\G_u}-1\\
	0&\frac{\delta(\delta-(\bar{\delta}_{\G_u}+1))}{v-(\delta+1)}&\delta-\frac{\delta(\delta-(\bar{\delta}_{\G_u}+1))}{v-(\delta+1)}
\end{array}
\right)
\]
and has the characteristic polynomial 
\[
\chi_M(x)=(\delta-x)\left(x^2+\frac{(\delta(\delta-1)-\bar{\delta}_{\G_u}(v-1))x+\delta(2\delta-\bar{\delta}_{\G_u}-v)}{v-\delta-1}\right).
\] 
The eigenvalues of $M$ are $\delta$, $F(\bar{\delta}_u)$ and 
\[
\tfrac{\bar{\delta}_u(v-1)-\delta(\delta-1)-\sqrt{(\bar{\delta}_u(v-1)-\delta(\delta-1))^2+4(v-\delta-1)\delta(v-2\delta+\bar{\delta}_u})}{2(v-\delta-1)}.
\]
Note that by Lemma \ref{Lemma Root} all  eigenvalues of $M$ are real.
By Theorem \ref{Corollary Haemers} the eigenvalues of $M$ interlace the eigenvalues of the adjacency matrix of ${\mathcal G}$, hence $\nu_2({\mathcal G})\geq F(\bar{\delta}_u)$.
This is true for any vertex $u$ with connected neighbourhood graph, in particular if the right hand side of the above inequality  is maximized.

\underline{Case 2: $\G_u$ is not connected}

For any connected component ${\mathcal C}(\G_u)$  we divide the adjacency matrix of ${\mathcal G}$ into $16$ block matrices $M_{ij}$ according to $u$ and the vertices of ${\mathcal C}(\G_u)$, ${\mathcal G}_u\setminus {\mathcal C}(\G_u)$ and ${\mathcal G}\setminus \{u,\G_u\}$. Let $\gamma$ denote the size of ${\mathcal C}(\G_u)$ and $\bar{\delta}_{{\mathcal C}(\G_u)}$ denote the average row sum of $M_{ij}$. Then the matrix of average row sums is
\[
M=\left(
\begin{array}[c]{cccc}
	0& \gamma& \delta-\gamma& 0\\
	1&\bar{\delta}_{{\mathcal C}(\G_u)}&0&\delta-(\bar{\delta}_{{\mathcal C}(\G_u)}+1)\\
	1&0&\bar{\delta}_{{\mathcal C}(\G_u)}&\delta-(\bar{\delta}_{{\mathcal C}(\G_u)}+1)\\
	0&\frac{\gamma(\delta-(\bar{\delta}_{{\mathcal C}(\G_u)}+1))}{v-(\delta+1)}&\frac{(\delta-\gamma)(\delta-(\bar{\delta}_{{\mathcal C}(\G_u)}+1))}{v-(\delta+1)}&\delta-\frac{\delta(\delta-(\bar{\delta}_{{\mathcal C}(\G_u)}+1))}{v-(\delta+1)}
\end{array}
\right)
\]
and has the characteristic polynomial 
\[
\chi_M(x)=(\bar{\delta}_{{\mathcal C}(\G_u)}-x)(\delta-x)\left(x^2+\frac{(\delta(\delta-1)-\bar{\delta}_{{\mathcal C}(\G_u)}(v-1))x+\delta(2\delta-\bar{\delta}_{{\mathcal C}(\G_u)}-v)}{v-\delta-1}\right).
\] 
By Theorem \ref{Corollary Haemers} the eigenvalues of $M$ interlace the eigenvalues of the adjacency matrix of ${\mathcal G}$, hence 
\begin{equation*} \nu_2(G)\geq\max\left\{\bar{\delta}_{{\mathcal C}(\G_u)},F(\bar{\delta}_{{\mathcal C}(\G_u)})\right\}.\end{equation*}
This is true for any connected component ${\mathcal C}(\G_u)$, in particular if the right hand side of the above inequality  is maximized.
\end{proof}

The following lemma gives us a little more information on how the function $F(x)$ in Proposition \ref{Proposition lower bound eigenvalue} behaves and what the values for $\eta_u$ are in the different cases.

\begin{lemma}\label{Lemma values}
Let $v\geq 3$, $1<\delta<v-1$  and  $F:{\mathbb R}_{\geq0}\rightarrow {\mathbb R}$,
\[
F(x)=\tfrac{x(v-1)-\delta(\delta-1)+\sqrt{(x(v-1)-\delta(\delta-1))^2+4(v-\delta-1)\delta(v-2\delta+x)}}{2(v-\delta-1)}.\]
Then 
\begin{itemize}
\item $F$ is increasing on ${\mathbb R}_{\geq0}$
\item $F$ is strictly convex on ${\mathbb R}_{\geq0}$
\item if $\delta=2$ then
\begin{itemize}
 \item if $v=4$,  $F(x)\leq x \text { if }x=0\text{ and }F(x)>x\text{ else;}$
\item $\text{ if }v>4,\ F(x)> x \text { for all }x\geq 0.$
\end{itemize}
\item if $2<\delta<v-1$ then 
\begin{itemize}
\item if $v\leq 2\delta$ then 
\[
F(x)\leq x \text{ for }x\in\left[0,\tfrac{1}{2}(\delta-2+\sqrt{(\delta+2)^2-4v})\right]
\] and $F(x)>x$ else;
\item if $2\delta<v<\tfrac{1}{4}(\delta+2)^2$ then $F(x)\leq x$ for
\[
x\in\left[-\tfrac{1}{2}(\delta-2-\sqrt{(\delta+2)^2-4v}),\tfrac{1}{2}(\delta-2+\sqrt{(\delta+2)^2-4v})\right]
\] and $F(x)>x$ else;
\item if $v=\tfrac{1}{4}(\delta+2)^2$ then
\[
F(x)\leq x\text{ for }x=\tfrac{1}{2}(\delta-2)
\] and $F(x)>x$ else;
\item if $v>\tfrac{1}{4}(\delta+2)^2$ then $F(x)>x$ for all $x\geq 0$.
\end{itemize}
\end{itemize}
\end{lemma}
\begin{proof}

Since 
\[
(x(v-1)-\delta(\delta-1))^2+4(v-\delta-1)\delta(v-2\delta+x)>0
\]for $1<\delta<v-1$ and $x\geq 0$,  and with Lemma \ref{Lemma Root} the derivative of $F$ is
\[
2(v-\delta-1)F'(x)=v-1+\tfrac{ (v-1) (x(v-1) -\delta(\delta-1) )+2 \delta (v-\delta-1)}{\sqrt{(x(v-1)-\delta(\delta-1))^2+4(v-\delta-1)\delta(v-2\delta+x)}}.
\]
Suppose there exists $\bar{x}>0$ such that $F'(\bar{x})\leq 0$, then
\begin{eqnarray*}
\sqrt{(\bar{x}(v-1)-\delta(\delta-1))^2+4(v-\delta-1)\delta(v-2\delta+\bar{x})}\leq - [(\bar{x}(v-1) -\delta(\delta-1) )+2 \delta (v-\delta-1)].
\end{eqnarray*}Squaring both sides and elimination of common terms gives
\[
-4\delta v(\delta+1-v)^3\leq 0,
\]a contradiction to $0<\delta<v-1$.
Therefore, $F'(x)> 0$ for all $x \in {\mathbb R}_{>0}$ and the map $F(x)$ is strictly increasing for all $x\geq 0$.

Since second derivative
\[
F''(x)=\frac{2v\delta(\delta+1-v)^2}{\left[(x(v-1)-\delta(\delta-1))^2+4(v-\delta-1)\delta(v-2\delta+x)\right]^{\tfrac{2}{3}}}.
\] is strictly positive for all $x\geq 0$  by Lemma \ref{Lemma Root},  it follows that $F(x)$ is strictly convex on ${\mathbb R}_{\geq0}$.

We want to solve the equation $F(x)=x$. This equation is satisfied if
\begin{equation*}
\sqrt{(x(v-1)-\delta(\delta-1))^2+4(v-\delta-1)\delta(v-2\delta+x)}=x(v-1-2\delta)+\delta(\delta-1).
\end{equation*}Squaring both sides and subtracting the right hand side gives
\begin{equation}\label{equation}
 4\delta(\delta+1-v)\left(x^2-(\delta-2)x-(2\delta-v)\right)=0.
\end{equation}
The equation has the solutions
\[
\tfrac{1}{2}\left(\delta-2\pm\sqrt{(\delta+2)^2-4v}\right).
\]First of all we note that for $v>\tfrac{1}{4}(\delta+2)^2$ and  $\delta\geq 2$ there are no solutions of equation \ref{equation} in ${\mathbb R}$, and therefore $F(x)>x$ for all $x\in{\mathbb R}_{\geq 0}$ in this case.

Suppose $4v\leq(\delta+2)^2$. 
We start with the case $\delta=2$: here, there are no solutions for $v>4$. If $v=4$, then the only solution of equation \ref{equation} is $x=0$ and indeed $F(0)=0$. 
Now let $\delta>2$. 
If $v\leq 2\delta$, then $(\delta+2)^2-4v\geq (\delta-2)^2$ and 
the positive solution of equation \ref{equation} is 
$\tfrac{1}{2}(\delta-2+\sqrt{(\delta+2)^2-4v}).$

If $2\delta<v<1/4(\delta+2)^2$, then $(\delta+2)^2-4v<(\delta-2)^2$ and 
the positive solutions of equation \ref{equation} are 
$
\tfrac{1}{2}(\delta-2-\sqrt{(\delta+2)^2-4v})$ and $\tfrac{1}{2}(\delta-2+\sqrt{(\delta+2)^2-4v}).$
If $v=\tfrac{1}{4}(\delta+2)^2$, then $\sqrt{(\delta+2)^2-4v}=0$ and the only solution to equation \ref{equation} is $x=\tfrac{1}{2}(\delta-2)$. 

It remains to show that $\tfrac{1}{2}(\delta-2\pm\sqrt{(\delta+2)^2-4v})$ are solutions to the equation $F(x)=x$. In fact, we show that $\tfrac{1}{2}(\delta-2+y)$ are solutions to $F(x)=x$ for  $y\in\{\pm\sqrt{(\delta+2)^2-4v}\}$.
As in Lemma \ref{Lemma Root}, let 
\[
f(x)=(x(v-1)-\delta(\delta-1))^2+4(v-\delta-1)\delta(v-2\delta+x).
\]Then
\begin{eqnarray*}
 4f(\tfrac{1}{2}\left(\delta-2+y\right))=
\left((v-1-2\delta)y+v(\delta-2)+\delta+2\right)^2 +4\delta(v-1-\delta)(y^2-(\delta+2)^2-4v),
\end{eqnarray*}
and for $y\in\{\pm\sqrt{(\delta+2)^2-4v}\}$  
\[
4f(\tfrac{1}{2}\left(\delta-2+y\right))=[(v-1-2\delta)y+v(\delta-2)+\delta+2]^2.
\]Therefore  
\begin{eqnarray*}
F\left(\tfrac{1}{2}\left(\delta-2+y\right)\right)&=&\frac{\tfrac{1}{2}(v-1)(\delta-2+y)-\delta(\delta-1)+\sqrt{f(\tfrac{1}{2}\left(\delta-2+y\right))}}{2(v-\delta-1)}\\
&=&\frac{\tfrac{1}{2}(v-1)(\delta-2+y)-\delta(\delta-1)+\tfrac{1}{2}\left((v-1-2\delta)y+v(\delta-2)+\delta+2\right)}{2(v-\delta-1)}\\
&=&\tfrac{1}{2}\left(\delta-2+y\right).
\end{eqnarray*}
Since $F$ is increasing and convex on ${\mathbb R}_{\geq0}$, for  solutions $ x_1\leq x_2$ of $F(x_i)=x_i$ for $i=1,2$, it follows $F(x)\leq x$ for $x\in[x_1,x_2]$ and $F(x)>x$ else. 
For $v\leq 2\delta$ there are solutions $x_1<0$ and $x_2\geq0$ for $F(x)=x$. But since $F$ is convex, it follows that $F(x)\leq x$ for $x\in[0,x_2]$.
\end{proof}


We will conclude this section with an example of a class of graphs that meet the bound in Proposition \ref{Proposition lower bound eigenvalue} showing that this bound can not be improved.

Let $v+1=2\delta$.
By the Handshaking Lemma, which states that a simple regular graph with odd degree must have an even number of vertices, $\delta$ must be even and therefore $v+1\equiv 0\mod 4$.
Let $\G$ be the graph obtained from the complete bipartite graph with parts of size $\frac{v-1}{2}$ and $\frac{v+1}{2}$ by adding $\frac{v+1}{4}$ disjoint edges to pair up the vertices of the part of size $\frac{v+1}{2}$. For example for $v=11$ and $\delta=6$ the graph is the following: 

\bigskip

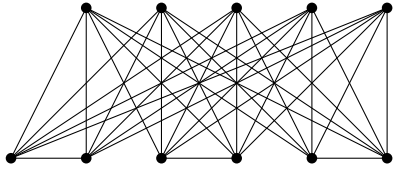
\begin{center}

\begin{tikzpicture}[rotate=90,every node/.style={fill,circle,inner sep=0pt,minimum size=4pt}]
\draw (0,1) node[]{}  -- (0,2) node[]{};
\draw (0,3) node[]{} -- (0,4) node[]{};
\draw (0,5) node[]{} -- (0,6) node[]{};
\draw (2,2) node[]{};
\draw (2,3) node[]{};
\draw (2,4) node[]{};
\draw (2,5) node[]{};
\draw (2,1) node[]{};
\foreach \x in {1,2,3,4,5,6}
{
\foreach \y in {1,2,3,4,5}{
\draw (0,\x)--(2,\y);
}
}
\end{tikzpicture} 
\end{center}
\bigskip

The second largest eigenvalue of its adjacency matrix is $1$ with mutliplicity $\geq 1$. The corresponding eigenvectors are the vectors of the form 
\[
(\underbrace{0,\cdots,0}_{\frac{v-1}{2}},-1,-1,1,1,0,\cdots,0),\ 
(\underbrace{0,\cdots,0}_{\frac{v-1}{2}},-1,-1,0,0,1,1,0,\cdots,0),\ 
\cdots
\]
For a vertex in the part of size $(v+1)/2$, the neighbourhood graph is always connected. In fact, it is the star graph with $(v-1)/2+1$ vertices. The average degree is $2(v-1)/(v+1)$. 
For a vertex in the part of  size $(v-1)/2$, the neighbourhood graph has $(v+1)/4$ components corresponding to the disjoint edges that pair up the vertices of the part of size $(v+1)/2$. The average degree of all these components is $1$. For $v=2\delta-1$ we have  $\tfrac{1}{2}(\delta-2+\sqrt{(\delta+2)^2-4v})\geq 1$ and therefore  $F(1)\leq 1$.
For $\delta>2$ also $F(2(v-1)/(v+1))\leq 1$ and Proposition \ref{Proposition lower bound eigenvalue} gives us the lower bound 
\[
\varrho(\G)=\max\left\{1,F(1),F(\tfrac{2(v-1)}{v+1})\right\}=1
\]
 on the second largest eigenvalue, which is tight in this case.

%
%

\section{A class of SRGs maximizing the algebraic connectivity}

\begin{lemma}
 Let ${\mathcal G}$ be a strongly regular graph with degree $\delta$ and parameters $\lambda$ and $\mu$. The neighbourhood graph $\G_u$  for any  vertex $u\in V(\G)$ is regular and has degree $\lambda$.
\end{lemma}
\begin{proof}
 Any vertex in $\G_u$ has exactly $\lambda$ common neighbours with $u$ since $\G$ is strongly regular.
\end{proof}

\begin{lemma}\label{Neighbourhood Proposition}Let ${\mathcal G}$ be a strongly regular graph  with degree $\delta$ and parameters $\lambda$ and $\mu$.
If \[
\lambda>\tfrac{1}{2}\left(\lambda-\mu+\sqrt{(\lambda-\mu)^2+4(\delta-\mu)}\right),
\] then the neighbourhood graph ${\mathcal G}_u$ of any vertex $u\in V(\G)$ is connected.
\end{lemma}
\begin{proof}
Suppose there is a vertex $u$ such that ${\mathcal G}_u$ is not connected and let ${\mathcal C}(\G_u)$ be a connected component of $\G_u$ of size $\gamma<\delta$. We divide the adjacency matrix of ${\mathcal G}$ into $16$ block matrices $M_{ij}$ according to $u$ and the vertices of ${\mathcal C}(\G_u)$, ${\mathcal G}_u\setminus {\mathcal C}(\G_u)$ and ${\mathcal G}\setminus \{u,\G_u\}$. Then the matrix of average row sums is
\[M=
\left(
\begin{array}[c]{cccc}
	0& \gamma& \delta-\gamma& 0\\
	1&\lambda&0&\delta-(\lambda+1)\\
	1&0&\lambda&\delta-(\lambda+1)\\
	0&\frac{\gamma(\delta-(\lambda+1))}{v-(\delta+1)}&\frac{(\delta-\gamma)(\delta-(\lambda+1))}{v-(\delta+1)}&\delta-\frac{\delta(\delta-(\lambda+1))}{v-(\delta+1)}
\end{array}
\right).
\]
With equation \ref{lemma parameters of SRG satisfy} we can write the characteristic polynomial of $M$ as
\[
\chi_M(x)=(\lambda-x)(\delta-x)(x^2+(\mu-\lambda)x+\mu-\delta).
\]The eigenvalues of $M$ therefore  are
\[
\delta,\lambda,\tfrac{1}{2}\left(\lambda-\mu\pm\sqrt{(\lambda-\mu)^2+4(\delta-\mu)}\right).
\]
By Theorem \ref{Corollary Haemers}, the eigenvalues of $M$ interlace the eigenvalues of the adjacency matrix of ${\mathcal G}$ and it follows directly that
\[
\nu_2(\G)=\tfrac{1}{2}\left(\lambda-\mu+\sqrt{(\lambda-\mu)^2+4(\delta-\mu)}\right)\geq \lambda.
\] 
\end{proof}
\begin{corollary}\label{Corollary lambda=}Let ${\mathcal G}$ be a strongly regular graph with degree $\delta$ and parameters $\lambda$ and $\mu$.
 If $\lambda=\nu_2(\G)$, 
then the size of every connected component of ${\mathcal G}_u$ is divisible by $\lambda+1$.
\end{corollary}
\begin{proof}
If  $\lambda=\nu_{2}(\G)= \tfrac{1}{2}\left(\lambda-\mu+\sqrt{(\lambda-\mu)^2+4(\delta-\mu)}\right)$,  then  $\lambda=\tfrac{\delta-\mu}{\mu}$, that is $\delta=\mu(\lambda+1)$. Let $u$ be a vertex of $\G$ and $\G_u$ the neighbourhood graph. Now, if $\G_u$ is connected, then the size of $\G_u$ is $\delta$ which is divisible by $\lambda+1$.
Suppose $\G_u$ is not connected. Then  we can divide the adjacency matrix of ${\mathcal G}$ into $16$ block matrices $M_{ij}$ according to $u$ and the vertices of a connected component ${\mathcal C}(\G_u)$ of ${\mathcal G}_u$ of size $\gamma<\delta$, ${\mathcal G}_u\setminus {\mathcal C}(\G_u)$ and ${\mathcal G}\setminus \{u,\G_u\}$. The $4\times 4$ matrix of the average row sums of these block matrices is exactly the matrix $M$ in the proof of Proposition \ref{Neighbourhood Proposition}.
Again, the eigenvalues of $M$ interlace the eigenvalues of the adjacency matrix of ${\mathcal G}$.
Since $\lambda$ is an eigenvalue of  ${\mathcal G}$,  the interlacing is tight. 
It follows with Theorem \ref{Corollary Haemers} that all  matrices $M_{ij}$ have constant row sums, in particular
\[
\frac{\gamma(\delta-\lambda-1)}{v-\delta-1}=\gamma\frac{\mu}{\delta}=\frac{\gamma}{\lambda+1}\in {\mathbb N},
\]
hence $\gamma$ is divisible by $\lambda+1$.
\end{proof}

\begin{theorem}\label{Proposition Optimal SRG}Let $v\geq 3$.
 Let $\G$ be a strongly regular graph with degree $\delta$ and parameters $\lambda$ and $\mu$ such that
\[
\lambda\geq \tfrac{1}{2}(\lambda-\mu+\sqrt{(\lambda-\mu)^2+4(\delta-\mu)}).
\]
  If $\lambda$ minimizes the average degree of the connected components of all neighbourhood graphs in a class of $\delta$-regular graphs on $v$ vertices, then $\G$ maximizes the algebraic connectivity among all graphs in that class.
\end{theorem}
\begin{proof}Let $u\in V(\G)$. Any connected component of $\G_u$ has average degree $\lambda$, since any neighbour of $u$ has exactly $\lambda$ common neighbours. Let $F:{\mathbb R}_{\geq0}\rightarrow {\mathbb R}$,
\[
F(x)=\tfrac{x(v-1)-\delta(\delta-1)+\sqrt{(x(v-1)-\delta(\delta-1))^2+4(v-\delta-1)\delta(v-2\delta+x)}}{2(v-\delta-1)}.\]
  With Proposition \ref{Proposition lower bound eigenvalue} it follows that for all $u\in V(\G)$
\begin{eqnarray*}
\xi_u&=&F(\lambda)\\
&=&\tfrac{1}{2}(\lambda-\mu+\sqrt{(\lambda-\mu)^2+4(\delta-\mu)})
\end{eqnarray*}
is a lower bound for the second largest eigenvalue of $\G$, and in fact we have equality.

 From Lemma \ref{Lemma values} follows that $F(x)$ is increasing for all $x\geq 0$  and if  $\lambda$  is the minimal average degree of the neighbourhood graph of any $\delta$-regular graph  in the class, then $F(x)$ attains its minimum at $x=\lambda$   among all regular graphs  in the class.  Note, that $F(\lambda)=\nu_2(\G)\leq \lambda$ by assumption.

Let $\G'$ be any $\delta$-regular graph and $\varrho(\G')$ the lower bound for the second largest eigenvalue $\nu_2(\G')$ of $\G'$ from Proposition \ref{Proposition lower bound eigenvalue}. Then  $\varrho(\G')$ is either an average degree $x\geq \lambda$ or  $\varrho(\G')=F(x)$. Since $x\geq \lambda\geq F(\lambda)$ and since $F(x)\geq F(\lambda)$ for all $x\geq \lambda$, the  eigenvalue $\nu_2(\G')$ is at least as large as $F(\lambda)=\nu_2(\G)$, the second largest eigenvalue of $\G$. Therefore, $\G$ maximizes $\delta-\nu_2(\G)$ among all $\delta$-regular graphs in the class.
\end{proof}

\begin{proposition}
Suppose there exists a SRG  with degree $\delta$ and parameters $\lambda$ and $\nu$. If $v\leq 2\delta-\lambda$ and  $\lambda\geq \tfrac{1}{2}(\lambda-\mu+\sqrt{(\lambda-\mu)^2+4(\delta-\mu)})$, then the SRG maximizes the algebraic connectivity among all $\delta$-regular graphs.
\end{proposition}

\begin{proof}
 Let $\G$ be any $\delta$-regular graph on $v$ vertices. For any vertex $u$ the neighbourhood graph $\G_u$ has $\delta$ vertices. Suppose $w$ is a vertex in $\G_u$ and its degree in the neighbourhood graph is $\delta_u(w)<\lambda$. Then $w$ has \[\delta-\delta_u(w)>\delta-\lambda\geq \delta-2\delta+v=v-\delta=|\G\setminus\G_u|\] neighbours outside $\G_u$, a contradiction. Therefore the conditions of Theorem \ref{Proposition Optimal SRG} are satisfied and the statement follows.
 \end{proof}
\begin{example}{Complete regular bipartite graphs}
Let $\alpha,m\in{\mathbb N}$ such that $\alpha\geq 2$ and $\alpha m\geq 4$.
The the complete regular multipartite  graph $K_{m,m,\ldots,m}$ with $\alpha$ parts of size $m$  is a strongly regular graph of degree $(\alpha-1)m$ and  parameters
\[
\lambda=(\alpha-2)m,\ \mu=(\alpha-1)m.
\]
Since 
\[
\lambda>0=\tfrac{1}{2}(\lambda-\mu+\sqrt{(\lambda-\mu)^2+4(\delta-\mu)})\text{ and }v=\alpha m=2\delta-\lambda,
\] the  conditions of the above proposition are satisfied. It follows the known fact (\cite{Takeuchi}) that  $K_{m,m,\ldots,m}$  maximizes the  the algebraic connectivity among all $(\alpha-1)m$-regular graphs on $\alpha m$ vertices. 
\end{example}

\addcontentsline{toc}{section}{Bibliography}

\bibliographystyle{agsm}
\bibliography{lit}


\end{document}